\documentclass [11pt,oneside,a4paper,mathscr]{amsart}

\usepackage{amscd,amsmath,amssymb,euscript,tabu,multirow,float,tikz-cd,tensind,slashbox,}
\tensordelimiter{?}
\usepackage[frame,cmtip,curve,arrow,matrix,line,graph]{xy}
\usepackage{xcolor}
\usepackage[colorlinks=true, linkcolor=., citecolor=blue]{hyperref}

\title[Stephen Pietromonaco]{$G$-invariant Hilbert Schemes on Abelian Surfaces and Enumerative Geometry of the Orbifold Kummer Surface}
\author{Stephen Pietromonaco}

\date{}

\newtheorem{thm}{Theorem}[section]
\newtheorem{conj}[thm]{Conjecture}
\newtheorem{cor}[thm]{Corollary}
\newtheorem{prop}[thm]{Proposition}
\newtheorem{lemma}[thm]{Lemma}

\newenvironment{pf}{\paragraph{Proof}}{\qed\par\medskip}

\renewcommand{\leq}{\leqslant}
\renewcommand{\geq}{\geqslant}

\theoremstyle{definition}
\newtheorem{defn}[thm]{Definition}
\newtheorem{rmk}[thm]{Remark}

\newcommand{\Pic}{\operatorname{Pic}}
\newcommand{\Jac}{\operatorname{Jac}}

\newcommand{\Aut}{\operatorname{Aut}}
\newcommand{\Autz}{\operatorname{Aut_{0}}}

\newcommand{\C}{\mathbb C}
\newcommand{\Hilb}{\operatorname{Hilb}}

\newcommand{\Q}{\mathbb Q}

\newcommand{\Z}{\mathbb Z}
\newcommand{\HH}{\mathbb H}
\newcommand{\RR}{\mathbb R}

\renewcommand{\P}{\mathbb{P}}

\newcommand{\SL}{\operatorname{SL}}

\renewcommand{\Im}{\operatorname{Im}}

\newcommand{\Chow}{\operatorname{Chow}}

\newcommand{\R}{\mathbb{R}}
\newcommand{\NS}{\operatorname{NS}}

\newcommand{\SU}{\operatorname{SU}}

\newcommand{\Mod}[1]{\ (\mathrm{mod}\ #1)}
\newcommand{\Hii}{\Hilb^{d}(A)^{G}}
\newcommand{\Part}{Z_{A,G}}
\newcommand{\Parti}{Z_{A,G}^{-1}}
\newcommand{\RelJac}{\overline{\mathcal{J}\text{ac}}^{\,\, d+1}_{\beta_{d}}(A)}
\newcommand{\CompJac}{\operatorname{\overline{Jac}}}
\newcommand{\Sing}{\operatorname{Sing}}
\newcommand{\Hf}{\mathfrak{H}}
\newcommand{\Zbir}{Z^{\textsf{bir}}}
\newcommand{\MT}{\mathsf{MT}_{y}}
\newcommand{\hdg}{\mathsf{h}_{d}(g)}
\newcommand{\ndh}{\mathsf{n}_{d}(h)}
\newcommand{\ndo}{\mathsf{n}_{d}(0)}

\newcommand{\Amod}{[A/\tau]}

\mathchardef\ordinarycolon\mathcode`\:
\mathcode`\:=\string"8000
\begingroup \catcode`\:=\active
  \gdef:{\mathrel{\mathop\ordinarycolon}}
\endgroup

\begin{document}

\begin{abstract}
For an Abelian surface $A$ with a symplectic action by a finite group $G$, one can define the partition function for $G$-invariant Hilbert schemes
\[Z_{A, G}(q) = \sum_{d=0}^{\infty} e(\text{Hilb}^{d}(A)^{G})q^{d}.\]
We prove the reciprocal $Z_{A,G}^{-1}$ is a modular form of weight $\frac{1}{2}e(A/G)$ for the congruence subgroup $\Gamma_{0}(|G|)$, and give explicit expressions in terms of eta products.  Refined formulas for the $\chi_{y}$-genera of $\text{Hilb}(A)^{G}$ are also given.  For the group generated by the standard involution $\tau : A \to A$, our formulas arise from the enumerative geometry of the orbifold Kummer surface $[A/\tau]$.  We prove that a virtual count of curves in the stack is governed by $\chi_{y}(\text{Hilb}(A)^{\tau})$.  Moreover, the coefficients of $Z_{A, \tau}$ are true (weighted) counts of rational curves, consistent with hyperelliptic counts of Bryan, Oberdieck, Pandharipande, and Yin.  
\end{abstract}

\maketitle


\section{Introduction}

\subsection{G\"{o}ttsche's Formula on the Orbifold \boldmath{$[A/G]$}}

In one of the seminal results of modern algebraic geometry, G\"{o}ttsche computed the Betti numbers of the Hilbert scheme of points $\Hilb^{d}(S)$ where $S$ is a smooth quasi-projective surface \cite{gottsche_betti_1990}.  When $S$ is a $K3$ surface, he found the remarkable formula
\begin{equation}\label{eqn:GottscheFormula}
\sum_{d=0}^{\infty} e( \Hilb^{d}(S))q^{d-1} = q^{-1} \prod_{n=1}^{\infty} (1-q^{n})^{-24}.
\end{equation}
The reciprocal of the righthand side equals the modular cusp form $\Delta(q) = \eta(q)^{24}$ of weight $12$.  

However, in the case of an Abelian surface $A$, G\"{o}ttsche's formula is trivial since $e(\Hilb^{d}(A))=0$ for $d>0$.  One can instead consider Abelian surfaces together with an action by a finite group.  By studying the invariant loci---or equivalently, working on the orbifold---we will find a family of results analogous to (\ref{eqn:GottscheFormula}).  

Consider a complex Abelian surface $A$, and a finite group $G$ acting on $A$ preserving the holomorphic symplectic form.  We will call such an action \emph{symplectic}.  The natural generating function produced from this data is
\begin{equation}\label{eqn:MainDefn}
\Part(q) := \sum_{d=0}^{\infty} e(\Hii)q^{d}
\end{equation}
where $\Hii$ is the $G$-invariant Hilbert scheme, parameterizing finite $G$-invariant subschemes of length $d$.  It is equivalently the fixed locus of the induced $G$ action on $\Hilb^{d}(A)$.  The $G$-invariant Hilbert scheme is disconnected, though each component is a smooth projective holomorphic symplectic variety of $K3$-type.

By definition of the orbifold $[A/G]$, we have
\[\Hilb(A)^{G} = \Hilb([A/G]).\]
So we regard $\Part(q)$ as analogous to the lefthand side of (\ref{eqn:GottscheFormula}) for $[A/G]$.  Our following result can be understood as the analogue of G\"{o}ttsche's formula for the orbifold $[A/G]$ (see Appendix \ref{modular} for details on the modular forms). 
\begin{thm}\label{thm:MainThm}
The function $\Parti(q)$ is a modular form of weight $\frac{1}{2}e(A/G)$ for the congruence subgroup $\Gamma_{0}(|G|)$.  Moreover, $\Parti$ is an explicit eta product (see Table \ref{table:TAB1} and Proposition \ref{prop:TransProp} below), and transforms with multiplier system induced from that of the Dedekind eta function.  It is a holomorphic, non-cuspidal form, normalized with leading coefficient 1.    
\end{thm}
\noindent Our proof of Theorem \ref{thm:MainThm} relies on recent methods of Bryan-Gyenge \cite{bryan_g-fixed_2020_TWO} in the case of $K3$ surfaces.

Fujiki has completely classified symplectic actions by finite groups on Abelian surfaces \cite{fujiki_finite_1988}.  In the case where the subgroup acting by translations is trivial, the only groups which arise, up to isomorphism, are
\[ \Z_{2}, \,\, \Z_{3}, \,\, \Z_{4}, \,\, \Z_{6}, \,\, \mathcal{Q}, \,\, \mathcal{D}, \,\, \mathcal{T}\]
where we denote by $\Z_{n}$ the cyclic group of order $n$, and by $\mathcal{Q}, \mathcal{D}$, $\mathcal{T}$ the quaternion group of order 8, the binary diherdral group of order 12, and the binary tetrahedral group of order 24, respectively.  Recall that these groups fall into the ADE classification: $\Z_{n}$ has ADE type $A_{n-1}$ while $\mathcal{Q}, \mathcal{D}, \mathcal{T}$ have types $D_{4}, D_{5}, E_{6}$, respectively.

First consider actions by group homomorphisms; we call these \emph{linear}.  All Abelian surfaces carry a unique linear action by $\langle \tau \rangle \cong \Z_{2}$ where $\tau : A \to A$ is the standard involution.  Hence, $A$ admits a symplectic linear $\Z_{3}$ action if and only if it does so for $\Z_{6}$.  Then it suffices to study $G$ isomorphic to one of $\Z_{4}, \Z_{6}, \mathcal{Q}, \mathcal{D}, \mathcal{T}$.  

\begin{rmk}
By the physical arguments of \cite[Sec.\ 4]{volpato_symmetries_2014_TWO}, we understand why precisely these groups arise.  Let $G$ be isomorphic to one of the five groups listed above, and let $\overline{G}= G/ \langle \tau \rangle$ be the quotient by the unique order $2$ subgroup.  Then the five $\overline{G}$ are precisely the subgroups of the even Weyl group $W^{+}(E_{8})$ of the $E_{8}$ root lattice which pointwise fix a lattice of rank at least $4$.  This is in close analogy with the classification in the case of $K3$ surfaces.  
\end{rmk}

Any group $G$ with a symplectic action on $A$ can be written uniquely as an extension
\[0 \to T \to G \to G_{0} \to 0\]
where $T \subseteq G$ is the subgroup of all elements acting by translation, and the quotient $G_{0}$ acts linearly and symplectically on $A$.  If $T$ is trivial, we say the $G$ action is \emph{translation-free}.  Note that $T \cong \Z_{a} \times \Z_{b}$ for some $a,b \geq 1$.  

In Table \ref{table:TAB1} we present the modular form $\Parti$ for all equivalence classes of translation-free actions.  For such actions, $G$ and $G_{0}$ are abstractly isomorphic.  However, $G$ might not act linearly.  Notice there are translation-free actions by $\mathcal{Q}$ and $\mathcal{T}$ without fixed points (Nos.\@ $8$ and $11$ in the table, respectively), so in particular, they do not preserve the origin.  

\begin{table}[h]
\centering
{\tabulinesep=1.6mm
\begin{tabu}{|| c | c | c | c | c ||}
\hline
No. & $G$ &  Singularities of $A/G$ & Modular form $\Parti$ &  $\frac{1}{2}e(A/G)$ \\
\hline \hline
1 & $\{e\}$ &  --- & 1 & 0\\
\hline
2 & $\Z_{2}$ &  $16A_{1}$ & $\frac{\eta^{16}(q)}{\eta^{8}(q^{2})}$ & 4 \\
\hline
3 & $\Z_{3}$ &  $9A_{2}$ & $\frac{\eta^{9}(q)}{\eta^{3}(q^{3})}$ & 3 \\
\hline
4 & $\Z_{4}$ &  $4A_{3} + 6A_{1}$ & $\frac{\eta^{6}(q^{2}) \eta^{4}(q) }{\eta^{4}(q^{4})}$ & 3 \\
\hline
5 & $\Z_{6}$ &  $A_{5}+4A_{2}+5A_{1}$ & $\frac{\eta^{5}(q^{3}) \eta^{4}(q^{2})  \eta(q) }{\eta^{4}(q^{6})}$ & 3 \\
\hline
6 & $\mathcal{Q}$ &  $2D_{4} + 3 A_{3} + 2A_{1}$ & $\frac{\eta^{8}(q^{4}) \eta^{2}(q)}{\eta^{4}(q^{8}) \eta(q^{2})}$ & $5/2$ \\
7 & $\mathcal{Q}$ &  $4D_{4} + 3 A_{1}$ & $\frac{\eta^{15}(q^{4}) \eta^{4}(q)}{\eta^{6}(q^{8}) \eta^{8}(q^{2})}$ & $5/2$ \\
8 & $\mathcal{Q}$ &  $6A_{3} + A_{1}$ & $\frac{\eta(q^{4}) \eta^{6}(q^{2})}{\eta^{2}(q^{8})}$ & $5/2$ \\
\hline
9 & $\mathcal{D}$ &  $D_{5}+3A_{3}+2A_{2}+A_{1}$ & $\frac{\eta^{3}(q^{6}) \eta^{3}(q^{4}) \eta^{3}(q^{3}) \eta(q)}{\eta^{3}(q^{12}) \eta^{2}(q^{2})}$ & $5/2$\\
\hline 
10 & $\mathcal{T}$ &  $E_{6}+D_{4}+4A_{2}+A_{1}$ & $\frac{\eta^{5}(q^{12}) \eta^{6}(q^{8}) \eta(q^{3}) \eta(q)}{\eta^{4}(q^{24}) \eta^{2}(q^{6}) \eta^{2}(q^{2})}$ & $5/2$\\
11 & $\mathcal{T}$ &  $A_{5}+2A_{3}+4A_{2}$ & $\frac{\eta^{4}(q^{8}) \eta^{2}(q^{6}) \eta(q^{4})}{\eta^{2}(q^{24})}$ & $5/2$\\
\hline
\end{tabu}}
\caption{The modular forms $\Parti(q)$ for symplectic, translation-free actions.  The weight of the modular form is $\frac{1}{2}e(A/G)$, which is presented in the last column.}
\label{table:TAB1}
\end{table}

This reduces the problem to computing $\Part$ when $T$ is non-trivial.  Interpreting $T$ as a subgroup of $A$ let $A'=A/T$, which is again an Abelian surface.  We then get a symplectic translation-free action of $G' = G/T$ on the Abelian surface $A'$.  In Section \ref{partitionfunc} we will prove the following result.\footnote{Note that $G_{0}$ and $G'$ are abstractly isomorphic.  But we distinguish them because they are different groups acting on different spaces.}  
\begin{prop}\label{prop:TransProp}
With the notation as above, we have
\[\Part(q) = Z_{A', G'}(q^{|T|}).\]
In particular, the modular form $\Parti$ where $G$ has translations is an oldform: it is equal to a modular form $Z_{A', G'}^{-1}$ from Table \ref{table:TAB1} with the variable change $q \mapsto q^{|T|}$. 
\end{prop}

\subsection{Refinement to $\chi_{y}$-genus}
We can refine our formulas by replacing the Euler characteristic with a more elaborate index.  For our purposes, we will focus on the (normalized) $\chi_{y}$-genus, which for a compact complex manifold $M$ is defined in terms of the Hodge numbers as
\begin{equation}\label{eqn:chiydefn}
\begin{split}
\overline{\chi}_{y}(M) & = (-y)^{-\frac{1}{2}\dim(M)} \chi_{y}(M)\\
& = (-y)^{-\frac{1}{2}\dim(M)} \sum_{p,q} (-1)^{p} h^{p,q}(M) y^{q}.
\end{split}
\end{equation}
Notice that setting $y=-1$ recovers the Euler characteristic, $\overline{\chi}_{-1}(M) = e(M)$.  

Our formulas will involve the function 
\begin{equation}\label{eqn:WJFphim21}
\phi_{-2,1}(q,y) = \big(y^{\frac{1}{2}}-y^{-\frac{1}{2}}\big)^2 \prod_{n=1}^{\infty} \frac{(1-yq^{n})^{2}(1-y^{-1}q^{n})^{2}}{(1-q^{n})^{4}}
\end{equation}
which is the unique weak Jacobi form of weight $-2$ and index $1$ \cite[Thm.\@ 9.3]{eichler_theory_1985}.  We define the generating function
\[Z^{\overline{\chi}}_{A,G}(q,y) = \sum_{d=0}^{\infty} \overline{\chi}_{y}(\Hilb^{d}(A)^{G})q^{d}.\]

\begin{prop}\label{prop:chiyrefinement}
For all non-trivial translation-free symplectic actions we have
\[Z^{\overline{\chi}}_{A,G}(q,y) = -\big(y^{\frac{1}{2}}+y^{-\frac{1}{2}}\big)^2 \frac{\Part(q)}{\phi_{-2,1}(q^{|G|},-y)}.\]
\end{prop} 

Following \cite{bryan_g-fixed_2020_TWO}, one can give similar formulas for the elliptic genus, the motivic class, and more generally, the birationality class, but we will not need those here.

\subsection{Enumerative Geometry of the Orbifold Kummer Surface}\label{IntroEnumGeom}

The Katz-Klemm-Vafa (KKV) formula \cite{katz_m_1999} was predicted by string theorists to compute the BPS states of D-branes moving in a $K3$ surface $S$.  

In its modern mathematical formulation, the Maulik-Toda proposal \cite{maulik_gopakumarvafa_2018} is applied to a (local) $K3$ surface to define BPS invariants $\mathsf{n}_{\beta}^{K3}(g)$ for each effective curve class $\beta$.  These quantities, which we interpret as virtual counts of curves of geometric genus $g$ in the class $\beta$, only depend on $\beta$ through the self-intersection $\beta^{2}=2d-2$, so we denote them by $\mathsf{n}_{d}^{K3}(g)$.  The KKV formula is then
\begin{equation}\label{eqn:KKVFormula}
\begin{split}
\sum_{d=0}^{\infty}\sum_{g=0}^{\infty} \mathsf{n}_{d}^{K3}(g)(y^{\frac{1}{2}}+y^{-\frac{1}{2}})^{2g}q^{d-1} & = - (y^{\frac{1}{2}}+y^{-\frac{1}{2}})^{2} \frac{1}{\Delta(q) \phi_{-2,1}(q,-y)}\\
& = \frac{1}{q} \prod_{n=1}^{\infty} \frac{1}{(1-q^{n})^{20}(1+yq^{n})^{2}(1+y^{-1}q^{n})^{2}}.
\end{split}
\end{equation}
The coefficient of $q^{d-1}$ in the formula is $\overline{\chi}_{y}(\Hilb^{d}(S))$, so the KKV formula relates the $\chi_{y}$-genera of $\Hilb(S)$ and virtual counts of curves on a $K3$ surface.  It has now been proven in full \cite{Pandharipande:2014qoa}.  

In this paper we prove an analogue of the KKV formula for the \emph{orbifold Kummer surface} $\Amod$ by introducing a notion of $\tau$-BPS states.  Here $A$ is a polarized Abelian surface of type $(1, d)$ with $\beta_{d} \in H_{2}(A, \Z)$ the class of the primitive polarization\footnote{Throughout, we must handle $d=0$ separately.  In this case, choose the product Abelian surface $A=E \times F$, with $\beta_{0}$ the class of $E \times \{\text{pt}\}$.} and 
\[
\tau: A \to A
\]
is the involution $a \mapsto -a$, which we will call the standard involution.  

As with ordinary BPS invariants, consider the moduli space $M_{A}(0, \beta_{d}, 1)$ of Simpson stable sheaves on $A$ with Chern character $(0, \beta_{d}, 1) \in H^{2*}(A,\Z)$, and generic polarization.  The $\tau$ action lifts canonically to $M_{A}(0, \beta_{d}, 1)$ by pullback.   

The following is the Abelian surface-version of the fact that for a $K3$ surface $S$, a moduli space of stable sheaves with primitive Mukai vector and generic polarization is deformation equivalent to a Hilbert scheme of points on $S$.  
\begin{prop}\label{prop:KeyEnummPropp}
There exists a $\tau$-equivariant deformation equivalence 
\begin{equation}\label{eqn:keyDefEq}
M_{A}(0, \beta_{d}, 1) \to \widehat{A} \times \Hilb^{d}(A)
\end{equation}
where $\tau$ acts on both sides by pullback, and $\widehat{A} = \Pic^{0}(A)$ is the dual Abelian variety.  
\end{prop}
\noindent This is essentially a result of Yoshioka \cite{yoshioka_moduli_2001}.  Our observation is simply that his correspondence is $\tau$-equivariant, and we prove this in Section \ref{proofofKeyEnummPropp}.  

An immediate corollary is the following.  
\begin{cor}\label{YoshCorr}
Restricting to the $\tau$-invariant locus, we get a component-wise deformation equivalence
\begin{equation}
M_{A}(0, \beta_{d}, 1)^{\tau} \to \coprod_{i=1}^{16} \Hilb^{d}(A)^{\tau}.
\end{equation}
In particular, each component of $M_{A}(0, \beta_{d}, 1)^{\tau}$ is a smooth holomorphic symplectic variety of $K3$-type.  
\end{cor}

The ordinary Hilbert-Chow morphism is $\tau$-equivariant, so we can restrict to the invariant locus
\[\pi_{d} : M_{A}(0, \beta_{d}, 1)^{\tau} \to \Chow_{\beta_{d}}(A)^{\tau}.\]
which is a disjoint union of Lagrangian fibrations.  In Section \ref{Yau-ZaslowCompJac} we apply the Maulik-Toda proposal\footnote{The definition of the BPS invariants by Maulik-Toda applies to Calabi-Yau threefolds.  But in the case of a local Calabi-Yau surface, the theory reduces to a theory of sheaves on the surface, see Section \ref{subsec:GVIntro}.  Our results are therefore intrinsic to $[A/\tau]$.} to this map in order to define $\tau$-BPS invariants $\ndh$ of $\Amod$ (see Definition \ref{defn:GVinvofOKS}).  We interpret $\ndh$ as the virtual number of $\tau$-invariant curves in the class $\beta_{d}$ in $A$, whose quotient has geometric genus $h$.  Equivalently, $\ndh$ is a virtual count of genus $h$ curves on the orbifold.  

\begin{rmk}\label{rmk:noTwSec}
Our results to follow are coarse in the sense that the invariants $\ndh$ do not individually track the geometric genus of the $\tau$-invariant curves in $A$.  This is because we do not fully probe the K-theory of the orbifold $\Amod$.  In work in progress with J. Bryan \cite{bryan_counting_nodate} we give the refined formula, as well as propose a general framework defining equivariant BPS invariants on a Calabi-Yau threefold with an involution.
\end{rmk}

The following is an analogue of the KKV formula for the orbifold Kummer surface. 
\begin{thm}\label{prop:KKVSYProp}
The $\tau$-BPS invariants $\ndh$ (Definition \ref{defn:GVinvofOKS}) are determined by
\begin{equation}
\begin{split}
\frac{1}{16}\sum_{d=0}^{\infty} \sum_{h=0}^{\infty} \ndh \big(y^{\frac{1}{2}}+y^{-\frac{1}{2}}\big)^{2h}q^{d} & = -\big(y^{\frac{1}{2}}+y^{-\frac{1}{2}}\big)^{2} \frac{ Z_{A,\tau}(q) }{ \phi_{-2,1}(q^{2}, -y)}\\
& = \prod_{n=1}^{\infty} \frac{(1-q^{2n})^{12}}{(1-q^{n})^{16}(1+q^{2n}y)^{2}(1+q^{2n}y^{-1})^{2}}.
\end{split}
\end{equation}
\end{thm}
We prove this in Section \ref{ShenYinproof} using the work of Shen-Yin \cite{shen_topology_2019} on perverse Hodge numbers of Lagrangian fibrations to relate the Maulik-Toda polynomial to $\overline{\chi}_{y}(\Hilb(A)^{\tau})$.  We then apply our Proposition \ref{prop:chiyrefinement} which determines the $\chi_{y}$-genera.   

\begin{table}[h]
\centering
{\tabulinesep=1.6mm
\begin{tabu}{|| c || c | c | c | c | c | c | c | c ||}
\hline
$\frac{1}{16}\ndh$ & $d=0$ & $1$ & $2$ &  $3$ & $4$ & $5$ & $6$ & 7 \\
\hline \hline
$h=0$ & $1$ & 16 & 144 & 960 & 5264 & 25056 & 106944 & 418176 \\
\hline
$1$ & $0$ & $0$ & -2 & -32 & -294 & -2016 & -11400 & -56000 \\
\hline
$2$ & $0$ & $0$ & 0 & 0 & 3 & 48 & 448 & 3136 \\
\hline
$3$ & $0$ & $0$ & 0 & 0 & 0 & 0 & -4 & -64 \\
\hline
$4$ & $0$ & $0$ & 0 & 0 & 0 & 0 & 0 & 0 \\
\hline
\end{tabu}}
\caption{}
\label{table:TAB3}
\end{table}

We now want to specialize to counting rational curves.  Making the specialization $y=-1$ in the KKV formula (\ref{eqn:KKVFormula}) results in
\[\sum_{d=0}^{\infty} \mathsf{n}_{d}^{K3}(0) q^{d-1} = \frac{1}{\Delta(q)}.\]
This is the Yau-Zaslow formula---one of the earliest and most foundational results in modern enumerative geometry \cite{yau_bps_1996}.  It is a relationship between rational curves in $K3$ surfaces, modular forms, and Hilbert schemes of points.  One remarkable feature of the Yau-Zaslow formula is that the invariants $\mathsf{n}_{d}^{K3}(0)$ give actual (not virtual) counts of rational curves.  For the mathematical formulation of the theory, see \cite{bryan_enumerative_1999, beauville_counting_1999_TWO, fantechi_euler_1999}.  

In this spirit, we give an enumerative interpretation of our partition function $Z_{A, \tau}$ from Table \ref{table:TAB1} where $G = \langle \tau \rangle \cong \Z_{2}$.  Consistent with the perverse sheaf and vanishing cycle formalism of Maulik-Toda, Katz had previously defined the genus zero BPS invariants as Behrend-weighted Euler characteristics of the Simpson stable moduli space \cite{katz_genus_2008}.  

In the case of $\Amod$, we apply this to the space $M_{A}(0, \beta_{d}, 1)^{\tau}$, which is a disjoint union of smooth holomorphic symplectic varieties of $K3$-type.  Therefore, the Behrend weighting is trivial, and
\begin{equation}
\ndo = e\big(M_{A}(0, \beta_{d}, 1)^{\tau} \big).
\end{equation}

The following is an analogue of the Yau-Zaslow formula for $\Amod$, which we will prove in Section \ref{ratlcurveproof}.

\begin{thm}\label{thm:MainEmThm}
The genus zero $\tau$-BPS invariants $\ndo$ are determined by
\begin{equation}
\frac{1}{16} \sum_{d=0}^{\infty} \ndo q^{d} = Z_{A,\tau}(q) = \prod_{n=1}^{\infty} \frac{(1-q^{2n})^{8}}{(1-q^{n})^{16}}.
\end{equation}
Moreover, if $A$ is a $(1,d)$-polarized Abelian surface of Picard rank one and $\beta_{d}$ is the unique primitive generator, then $\ndo$ is a weighted count of rational curves on $\Amod$.  Specifically,
\begin{equation}\label{eqn:compJacform}
\ndo = \sum_{C \in \Pi} e(\CompJac (C)^{\tau})
\end{equation}
where $\Pi$ is the finite set of $\tau$-invariant curves in the class $\beta_{d}$ with rational quotient, and $\CompJac(C)$ is the compactified Jacobian of the integral curve $C$.  
\end{thm}
Just as the Yau-Zaslow formula does for $K3$ surfaces, this theorem relates rational curves on the orbifold $[A/\tau]$ to the modular form $Z_{A, \tau}^{-1}$ and hence, to Hilbert schemes of points on $[A/\tau]$ as well.

We note that in the case of a $K3$ surface with a symplectic action by a finite cyclic group, in \cite{zhan_counting_2021} they are interested in enumerating orbits of curves with rational quotient. Though the perspective is somewhat different from us: they view Euler characteristics as a representation, whereas we take ordinary Euler characteristics of the invariant locus.

\subsection{Hyperelliptic Curve Counting Invariants}
Our formula can equivalently be interpreted as a weighted count of hyperelliptic curves.  A study of the enumerative geometry of hyperelliptic curves in Abelian surfaces was initiated in \cite{rose_counting_2012}, and fully solved in \cite{bryan_curve_2018}. 
\begin{defn}
For a polarized Abelian surface $A$ of type $(1,d)$, let $\hdg$ be the finite number of geometric genus $g$ hyperelliptic curves in the class of the polarization such that all Weierstrass points lie at a $2$-torsion point of $A$.  
\end{defn}  
\begin{rmk}\label{rmk:CountRMK}
Every hyperelliptic curve in $A$ can be translated so that all Weierstrass points lie at $2$-torsion points, in which case the curve becomes $\tau$-invariant.  Therefore, $\hdg$ are actual (not virtual) counts of $\tau$-invariant curves with rational quotient.  
\end{rmk}
By \cite[Prop.\ 4]{bryan_curve_2018}, the $\hdg$ are computed from the formula\footnote{In \cite[Sec. 5.4]{bryan_curve_2018} what we are calling $\hdg$ was denoted $\textsf{h}_{g, \beta}^{A,\Hilb}$.}
\begin{equation}\label{eqn:BOPYref}
\sum_{d=0}^{\infty} \sum_{g=1}^{\infty} \hdg (w^{\frac{1}{2}}+w^{-\frac{1}{2}})^{2g+2} q^{d} = 4 \, \phi^{2}_{-2, 1}(q,-w)
\end{equation}
where $\phi_{-2,1}$ is the weak Jacobi form defined in (\ref{eqn:WJFphim21}).  In \cite{bryan_curve_2018}, it is tacitly assumed that $d >0$, but the formula correctly encodes the remaining invariant.  If $d=0$, the only non-vanishing invariant obtained from the formula is $\textsf{h}_{0}(1)=4$ which represents the four invariant genus one curves in the surface $A = E \times F$ in the class of $E \times \{ \text{pt} \}$, each of which have rational quotient.    

The relationship between $Z_{A, \tau}$ and the formula (\ref{eqn:BOPYref}) arises by making the specialization $w=1$.  We have the straightforward identity of infinite products
\[\frac{ \eta(q^{2})^{8}}{\eta(q)^{16}} = \frac{1}{16} \phi_{-2, 1}^{2}(q, -1)\]
and note that the lefthand side is precisely $Z_{A, \tau}$.  This along with (\ref{eqn:BOPYref}) and the first claim of Theorem \ref{thm:MainEmThm} immediately implies the following result.  
\begin{prop}\label{prop:HilbBOPY}
For all $d \geq 0$, we have
\[\ndo = \sum_{g=1}^{d+1} \hdg 2^{2g}.\]
\end{prop}
\noindent The invariants $\ndo$ are therefore less refined, as they do not individually track the geometric genus $g$ (see Remark \ref{rmk:noTwSec}).  

If $A$ has Picard rank one (except for the case of $d=0$) then by (\ref{eqn:compJacform}) we have
\[
\sum_{C \in \Pi} e(\CompJac (C)^{\tau}) = \sum_{g=1}^{d+1} \hdg 2^{2g}
\]
But $| \Pi | = \sum_{g} \hdg$, so we regard this as strong evidence for the following conjecture.
\begin{conj}\label{conj:CompJacConj}
If $A$ is an Abelian surface of Picard rank one, and $C \subset A$ is an integral $\tau$-invariant curve of geometric genus $g$ with rational quotient, then 
\[
e(\CompJac (C)^{\tau}) = 2^{2g}.
\]  
\end{conj}
\noindent In Section \ref{subsection:pf of conj for sm curves} we will prove that the conjecture holds in the case of smooth curves.  
 
\begin{rmk}
One should ask if there are similar enumerative interpretations of $e(\Hii)$ for the remaining groups in Table \ref{table:TAB1}.  To our knowledge, this breaks down outside of $G \cong \Z_{2}$ because the key deformation equivalence (\ref{eqn:keyDefEq}) is not $G$-equivariant.  Of course, one can directly study $e(M_{A}(0, \beta_{d}, 1)^{G})$ for the remaining $G$, but we do not pursue that here.  
\end{rmk}

\subsection*{Acknowledgements} I would like to thank my advisor Jim Bryan for his guidance and many helpful suggestions throughout this project, as well as the anonymous referee for their important feedback.


\section{Symplectic Actions on Abelian Surfaces}\label{FujikiSec}

\subsection{Preliminaries}\label{sympactions}
Let $X$ be a complex torus of arbitrary dimension, with $0 \in X$ the origin.  The group of biholomorphisms from $X$ to itself is denoted $\Aut(X)$, while the subgroup of linear maps (automorphisms of $X$ as a complex Lie group) is denoted $\Autz(X) \subset \Aut(X)$.  Given any $x \in X$, let $t_{x}: X \to X$ be the biholomorphism translating by $x$.

Given any holomorphic map $f: X \to X'$ between complex tori, using that $f$ is equivalent to a map between the corresponding universal covers, one can show that $h := t_{-f(0)} \circ f$ is linear.  Therefore, holomorphic maps between complex tori can be uniquely factored as a linear map composed with a translation 
\[f = t_{f(0)} \circ h.\]
This factorization induces a canonical surjective group homomorphism 
\[\sigma : \Aut(X) \to \Autz(X)\]
mapping $f$ to $h$, which restricts to the identity on $\Autz(X) \subset \Aut(X)$ and whose kernel is the subgroup of all translations of $X$.  This proves the following.
\begin{prop}\label{autdec}
The biholomorphism group of a complex torus $X$ decomposes as $\Aut(X) = \Autz(X) \rtimes X$, where $X$ is identified with the subgroup of translations.  
\end{prop}

Given a subgroup $G \subseteq \Aut(X)$, we get an action of $G$ on $X$ by biholomorphisms.  We will consider actions up to the following equivalence condition.
\begin{defn}
Consider pairs $(X_{i}, G_{i})$ for $i=1,2$ with $G_{i} \subseteq \Aut(X_{i})$.  We say the two pairs are \emph{equivalent} if there exists a biholomorphism $w:X_{1} \to X_{2}$ such that $G_{2} = w G_{1} w^{-1}$ in $\Aut(X_{2})$.  
\end{defn}

We specialize to the case where $X$ is a 2-dimensional complex torus.  A non-zero class $\alpha \in H^{2,0}(X)$ is called a \emph{holomorphic symplectic form}.  Since $h^{2,0}(X)=1$, a holomorphic symplectic form is unique up to scale.  
\begin{defn}
An automorphism $f \in \Aut(X)$ is holomorphic symplectic, or just \emph{symplectic}, if $f$ preserves a holomorphic symplectic form $\alpha$.  That is, if $f^{*} \alpha = \alpha$.  An action by a group $G$ on $X$ is \emph{symplectic} if each element of $G$ defines a symplectic automorphism.    
\end{defn}
\begin{lemma}\label{extensionlem}
A group $G$ with a symplectic action on $X$ can be written uniquely as the (possibly non-split) extension
\begin{equation}\label{eqn:SESsigma}
0 \to T \to G \to \sigma(G) \to 0
\end{equation}
where $T \subseteq G$ is the subgroup of all elements acting by translation, and the induced action by $\sigma(G)$ is symplectic and linear.
\end{lemma}
\begin{pf}
The existence of the short exact sequence is clear from the definition of $\sigma$.  Given $f \in G$ we can write $f=t_{f(0)} \circ h$, and then $\sigma(f) = h$, which acts linearly.  The kernel is precisely elements of $G$ acting by translation.  Since the symplectic form can be taken to be constant---induced from $dz_{1}dz_{2}$ on the universal cover---it is clearly invariant under translations.  Therefore, $h = t_{-f(0)} \circ f$ is symplectic.   
\end{pf}

This proof illustrates the obstruction to the splitting of the extension.  We have the unique factorization $f=t_{f(0)} \circ h$, but notice $t_{f(0)}$ might not be an element of $G$.  The extension splits if and only if $t_{f(0)} \in G$ for all $f \in G$.  

When $T$ is trivial, we say the action is \emph{translation-free}.  Note that a translation-free action is not necessarily linear: if $T$ is trivial, $G$ and $\sigma(G)$ are abstractly isomorphic, but may act differently on $X$.  We say that a translation-free action by $G$ is \emph{maximal} if there does not exist a translation-free group $H \subset \Aut(X)$ with $G \subsetneq H$ such that the $G$ action is the restriction of the $H$ action.

\subsection{Fujiki's Classification}

Fujiki has given a complete classification of symplectic actions by finite groups on two-dimensional complex tori \cite{fujiki_finite_1988}\footnote{This was equivalently carried out from a physics perspective in \cite{volpato_symmetries_2014_TWO} by studying symmetry groups of certain non-linear sigma models on the underlying real torus $T^{4}$.}.  The goal of this section is to condense the relevant results of Fujiki into a brief survey.  Below, all actions are assumed to be symplectic, and we will narrow our focus to Abelian surfaces, even though the results apply also to non-algebraic tori.  

By Lemma 3.3 of \cite{fujiki_finite_1988}, the only groups with non-trivial linear actions on an Abelian surface $A$ are (see Introduction for definitions) isomorphic to one of
\[ \Z_{2}, \,\, \Z_{3}, \,\, \Z_{4}, \,\, \Z_{6}, \,\, \mathcal{Q}, \,\, \mathcal{D}, \,\, \mathcal{T}.\]
All Abelian surfaces carry a unique linear action by $\langle \tau \rangle \cong \Z_{2}$ where $\tau : A \to A$ is the standard involution.  Hence, $A$ admits a symplectic linear $\Z_{3}$ action if and only if it does so for $\Z_{6}$.  Then it suffices to study $G$ isomorphic to one of $\Z_{4}, \Z_{6}, \mathcal{Q}, \mathcal{D}, \mathcal{T}$.  If we let $\overline{G}=G/\langle \tau \rangle$, then we know from \cite[Sec.\ 4]{volpato_symmetries_2014_TWO} that the five isomorphism classes of $\overline{G}$ are precisely all subgroups of $W^{+}(E_{8})$ which pointwise fix a lattice of rank at least $4$.  

One should understand which Abelian surfaces carry a linear action by a particular group.  For the two cyclic groups $\Z_{4}$ and $\Z_{6}$, the full description of which tori admit such actions is given in \cite[Prop.\@ 3.7]{fujiki_finite_1988}.  In particular, for all elliptic curves $E$, the product $E \times E$ admits an action by $\Z_{4}$ and $\Z_{6}$.  These linear actions by cyclic groups, including $\Z_{2}$ and $\Z_{3}$, correspond to Nos.\@ 2$-$5 in Table \ref{table:TAB1}.  

We now discuss the non-cyclic case.  Let $\HH \cong \RR^{4}$ denote the real quaternions.  The space of complex structures on $\HH$ is
\[M = \big\{ J \in \HH \, \big| \, J^{2} = -1 \big\} \cong \P^{1}\]
and can be identified with the imaginary unit quaternions.  Fujiki associates to $G \cong \mathcal{Q}, \mathcal{D}, \mathcal{T}$ a lattice $\Lambda_{G} \subset \HH$ and forms the real torus $T_{G} = \HH/\Lambda_{G}$.  The space $M$ parameterizes complex structures on $T_{G}$ such that the group of units $\Lambda_{G}^{\times} \cong G$ induces a holomorphic linear $G$ action on the complex torus.  The lattices and their groups of units can be found in (and just above) Lemma 2.6 of Fujiki.  

By Theorem 3.11 of Fujiki, all maximal linear actions by each $G$ arise in this way, up to equivalence.  These correspond to Nos.\@ 6, 9, and 10 in Table \ref{table:TAB1}.  There is a non-maximal linear $\mathcal{Q}$ action corresponding to the restriction of the maximal $\mathcal{T}$ action to the unique normal subgroup $\mathcal{Q} \subset \mathcal{T}$.  This is No.\@ 7 in the Table.  

In a very similar manner, Section 3.4 of Fujiki describes and classifies all non-linear translation-free actions.  Only $\mathcal{T}$ can act maximally as such, which corresponds to No.\@ 11 in the Table.  But we can restrict to $\mathcal{Q} \subset \mathcal{T}$ giving a non-maximal non-linear action of $\mathcal{Q}$.  This is No.\@ 8 in the Table.  

In all cases, the complex tori admitting a translation-free action by a non-cyclic group are parameterized by $M$.  Those that are algebraic are of a special form.  
\begin{defn}
A \emph{singular Abelian surface} is an Abelian surface $A$ whose Neron-Severi lattice $\NS(A)$ has rank 4, its largest possible value.  Equivalently, $A$ is a product $E \times F$ of isogenous elliptic curves with complex multiplication.  
\end{defn} 
The following result combines Lemma 5.6 and Proposition 5.7 of Fujiki.
\begin{prop}\label{FujikiPropp}
If $A$ admits a translation-free action by $G \cong \mathcal{Q}, \mathcal{D}, \mathcal{T}$ then $A$ is a singular Abelian surface.  Moreover, $A$ corresponds to a complex structure $J \in M$ such that $\mu J \in \Lambda_{G}$ for some real number $\mu \neq 0$, which depends on $J$.  
\end{prop}
Because $J$ must have unit norm, the set of $J \in M$ satisfying the second condition of Proposition \ref{FujikiPropp} is countable and dense in $M$.  Therefore, Abelian surfaces carrying an action by $\mathcal{Q}, \mathcal{D}$, and $\mathcal{T}$ are rigid---there are no infinitesimal deformations of the surface on which the group acts.  In Theorems 7.2 and 7.4 of Fujiki, necessary and sufficient conditions are given for a singular Abelian surface to admit an action by one of the three non-cyclic groups.

\subsection{Singularity Type}
Let $G$ be a finite group with a symplectic translation-free action on an Abelian surface $A$.  The singularities of $A/G$ are all of ADE type---that is, the stabilizer of an arbitrary point in $A$ is a finite subgroup of $\SL_{2}(\C)$.  Recall that $\Z_{n}$ has ADE type $A_{n-1}$ while $\mathcal{Q}, \mathcal{D}, \mathcal{T}$ have types $D_{4}, D_{5}, E_{6}$, respectively.  For a given action, let $a_{k}$ denote the number of $A_{k}$ singularities in $A/G$, let $d_{k}$ denote the number of $D_{k}$ singularities, and let $e_{k}$ denote the number of $E_{k}$ singularities.  In the present case, we present the \emph{singularity type} of an action as
\[a_{1}A_{1} + a_{2}A_{2} + a_{3}A_{3}+a_{5}A_{5} + d_{4}D_{4} + d_{5}D_{5} + e_{6}E_{6}.\]
\begin{prop}\label{prop:SingTypeProp}
The singularity type of a symplectic translation-free action of $G$ on $A$ is precisely one of those listed in Table \ref{table:TAB1}, with the corresponding value of $\frac{1}{2}e(A/G)$ listed in the final column.  
\end{prop}

We will sketch the proof of this proposition, but see also Lemma 3.19 in \cite{fujiki_finite_1988}\footnote{We claim there are a few minor but relevant typos in Lemma 3.19 and equation (19) of \cite{fujiki_finite_1988}.  Nonetheless, the ten actions described in Lemma 3.19 are precisely the ten non-trivial actions in our Table \ref{table:TAB1}.}.  Let $\Sing(A/G)$ and $(A/G)^{\circ} = A/G - \Sing(A/G)$ be the singular locus and smooth locus of $A/G$, respectively.  If $\pi: Y \to A/G$ is the minimal resolution of singularities, then by standard properties of the Euler characteristic, we have
\[e(Y) = e\big((A/G)^{\circ}\big) + e\big(\pi^{-1}(\Sing(A/G))\big).\]
The first term is computed by noting that $G$ acts freely on $A$ away from points with stabilizers
\begin{equation}\label{eqn:smLocEulChar}
e\big((A/G)^{\circ}\big) = -\bigg( \frac{a_{1}}{2} + \frac{a_{2}}{3}+\frac{a_{3}}{4}+\frac{a_{5}}{6}+\frac{d_{4}}{8}+\frac{d_{5}}{12}+\frac{e_{6}}{24}\bigg)
\end{equation}
and since the exceptional locus of $\pi: Y \to A/G$ is a disjoint union of ADE configurations of smooth rational curves, the second term is
\[e\big(\pi^{-1}(\Sing(A/G))\big) = 2a_{1} + 3a_{2} + 4a_{3} +6a_{5} + 5d_{4} + 6d_{5} + 7e_{6}.\]
Finally, since $Y$ is a smooth $K3$ surface, $e(Y)=24$.  We therefore get a strong numerical constraint on the numbers of singular points of each type.  For a given $G$, using that all subgroups define symplectic translation-free actions, one can systematically use this constraint to determine all allowed singularity types.  The possible solutions correspond to the 11 columns in Table \ref{table:TAB1}.      

Given a singularity type, let $r$ be the total number of singular points in $A/G$.  From (\ref{eqn:smLocEulChar}) and the obvious formula 
\[e(A/G) = e((A/G)^{\circ}) + r,\]
we can easily verify the values in the final column of Table \ref{table:TAB1}.


\section{Computation of the Partition Functions $\Part$}\label{partitionfunc}

The goal of this section is to compute the partition functions
\[\Part(q) := \sum_{d=0}^{\infty} e(\Hii)q^{d}\]
explicitly as an eta product, thereby proving Theorem \ref{thm:MainThm}.  First, we handle the more elementary Proposition \ref{prop:TransProp} on how translations affect the partition function.
\begin{proof}[Proof of Proposition \ref{prop:TransProp}]
If $G$ is a finite group acting symplectically on an Abelian surface $A$, and $T \subset G$ is the subgroup of translations, then $G'=G/T$ acts symplectically on $A'=A/T$ without translations.  Since $T$ acts freely, a $G$-invariant subscheme of $A$ must have length dividing $|T|$.  So $\Hii = \varnothing$ unless $d = m |T|$ for some integer $m \geq 0$, in which case $\Hilb^{m |T|} (A)^{G} \cong \Hilb^{m}(A')^{G'}$.  Hence,
\begin{equation*}
\begin{split}
\Part(q) = \sum_{d=0}^{\infty} e(\Hii)q^{d} & = \sum_{m=0}^{\infty} e(\Hilb^{m |T|}(A)^{G})q^{m |T|} \\
& = \sum_{m=0}^{\infty}e(\Hilb^{m}(A')^{G'})q^{m |T|} = Z_{A',G'}(q^{|T|}).
\end{split}
\end{equation*}
\end{proof}
Therefore, the problem is reduced to translation-free actions.  Our proof of Theorem \ref{thm:MainThm} is based on a computation of Bryan-Gyenge \cite{bryan_g-fixed_2020_TWO}, so we briefly review the relevant results.  For a finite subgroup $G_{\Delta} \subset \SU(2)$ with associated ADE root system $\Delta$, we can consider the natural action of $G_{\Delta}$ on $\C^{2}$, and define the local $G_{\Delta}$-fixed partition function
\[ Z_{\Delta}(q) = \sum_{n=0}^{\infty} e(\Hilb^{n}(\C^{2})^{G_{\Delta}})q^{n-\frac{1}{24}},\]
which Bryan-Gyenge compute for each allowed $\Delta$.

\begin{thm}[{\cite[Thm.\@ 1.2]{bryan_g-fixed_2020_TWO}}]\label{Bryan-Gynge}
For $\Delta$ of type $A_{n}$, the local partition function is given by
\[Z_{A_{n}}(q) = \frac{1}{\eta(q)}\]
while for type $D_{n}$ and $E_{n}$, it is
\[Z_{\Delta}(q) = \frac{\eta^{2}(q^{2}) \eta(q^{4E})}{\eta(q) \eta(q^{2E}) \eta(q^{2F}) \eta(q^{2V})}\]
with $(E, F, V)$ presented explicitly in each case in \cite[Thm.\@ 1.2]{bryan_g-fixed_2020_TWO}.
\end{thm}
With this, we are ready to prove our main result.  

\begin{proof}[Proof of Theorem \ref{thm:MainThm}]
Suppose first that $A$ is an Abelian surface with a translation-free symplectic action by a finite group $G$.  To set notation, let $p_{1}, \ldots, p_{r} \in A/G$ be the singular points, each with stabilizer subgroup $G_{i} \subset G$ and corresponding ADE root system $\Delta_{i}$.  With $k=|G|$ and $k_{i} = |G_{i}|$, let $\{x_{i}^{1}, \ldots, x_{i}^{k/k_{i}}\}$ be the orbit in $A$ corresponding to singular point $p_{i}$.  The smooth locus of the quotient is
\[(A/G)^{\circ} = A/G - \{p_{1}, \ldots, p_{r}\}.\]
The same method from \cite[Section 2]{bryan_g-fixed_2020_TWO} of stratifying the Hilbert scheme applies here, and gives the relation
\begin{equation}\label{eqn:MainComp}
\begin{split}
\sum_{d=0}^{\infty} e(\Hii)q^{d} = & \bigg( \sum_{d=0}^{\infty} e( \Hilb^{d}((A/G)^{\circ}))q^{kd}\bigg) \\
& \cdot \prod_{i=1}^{r} \bigg(\sum_{d=0}^{\infty} e(\Hilb^{d}(\C^{2})^{G_{i}})q^{\frac{dk}{k_{i}}}\bigg)
\end{split}
\end{equation}

By removing the singular points from $A/G$ as well as their preimages in $A$, the restricted quotient map is unramified of degree $k$, which means
\[e(A) - \sum_{i=1}^{r} \#\{x_{i}^{1}, \ldots, x_{i}^{k/k_{i}}\} = k \cdot e((A/G)^{\circ}).\]
With $a :=e((A/G)^{\circ})$, and using that $e(A)=0$, we get
\[a = - \sum_{i=1}^{r} \frac{1}{k_{i}}.\]

Since $(A/G)^{\circ}$ is a smooth quasi-projective surface, by G\"{o}ttsche's formula \cite{gottsche_betti_1990}
\[\sum_{d=0}^{\infty} e( \Hilb^{d}((A/G)^{\circ}))q^{kd} = \prod_{n=1}^{\infty} (1-q^{kn})^{-a}.\]
Hence, using the definition of the local partition functions, as well as (\ref{eqn:MainComp}), we get
\[
\begin{split}
\Part(q) & = \prod_{n=1}^{\infty} (1-q^{kn})^{-a} \cdot \prod_{i=1}^{r} q^{\frac{k}{24 k_{i}}}Z_{\Delta_{i}}\big(q^{\frac{k}{k_{i}}}\big) \\
& = q^{(\frac{ka}{24} + \sum_{i} \frac{k}{24k_{i}})} \eta^{-a}(q^{k}) \cdot \prod_{i=1}^{r} Z_{\Delta_{i}}\big(q^{\frac{k}{k_{i}}}\big).
\end{split}
\]
From the relation $a + \sum_{i} \frac{1}{k_{i}}=0$, we see the exponent of the overall power of $q$ vanishes.  With the substitution $a = e(A/G) -r$ we see
\begin{equation}\label{eqn:FinalExpression}
\Part(q) = \eta(q^{k})^{r-e(A/G)} \cdot \prod_{i=1}^{r} Z_{\Delta_{i}}\big(q^{\frac{k}{k_{i}}}\big).
\end{equation}

By Proposition \ref{prop:SingTypeProp}, the ADE singularity type and Euler characteristic $e(A/G)$ can be read off of Table \ref{table:TAB1}.  Recall that $\Z_{n}$ has ADE type $A_{n-1}$ while $\mathcal{Q}, \mathcal{D}, \mathcal{T}$ have types $D_{4}, D_{5}, E_{6}$, respectively.  This determines the value of $k_{i}$ for each $\Delta_{i}$.  From (\ref{eqn:FinalExpression}), we can use Theorem \ref{Bryan-Gynge} to compute the function $\Parti(q)$ as an eta product, which we record in the third column of Table \ref{table:TAB1}.  By Theorem \ref{Bryan-Gynge}, $Z_{\Delta_{i}}^{-1}$ transforms as a modular form of weight $\frac{1}{2}$ for all $\Delta_{i}$.  By (\ref{eqn:FinalExpression}), the weight of $\Parti$ is therefore $\frac{1}{2}e(A/G)$.  It is clear that the leading coefficient of $\Parti$ is $1$.  Applying Proposition \ref{prop:KohlerProp} case by case, we see that $\Parti$ is a holomorphic modular form of level $k = |G|$.  

In case $G$ acts on $A$ with translations, we apply Proposition \ref{prop:TransProp}.  Since the $G'$ action on $A'$ is symplectic and translation-free, $Z_{A',G'}^{-1}$ is a holomorphic modular form of weight $\frac{1}{2} e(A'/G')$, with level $|G'|$, and normalized with leading coefficient $1$.  The weight, holomorphy, and normalization are invariant under the variable change $q \mapsto q^{|T|}$, but the new level is $|G'| \cdot |T| = |G|$.
\end{proof}

\begin{proof}[Proof of Proposition \ref{prop:chiyrefinement}]
The proof of the much more general Theorem 1.10 in \cite[Sec.\ 5]{bryan_g-fixed_2020_TWO} goes through nearly verbatim in the case of Abelian surfaces, with the following minor adjustments.  Let $Y \to A/G$ be the minimal resolution.  We would use the definition $\Zbir_{A,G}(q) = \sum_{d=0}^{\infty} [\Hilb^{d}(A)^{G}]_{\textsf{bir}} \, q^{d}$.  But since $Y$ is a $K3$ surface, we define $\Zbir_{Y}(q)$ just as in \cite{bryan_g-fixed_2020_TWO} with the extra power of $q$.  In particular, this implies the result for the normalized $\chi_{y}$-genera.      
\end{proof}


\section{Enumerative Geometry of the Orbifold Kummer Surface}\label{Yau-ZaslowCompJac}

The purpose of this section is to prove the remaining results from Section \ref{IntroEnumGeom} of the Introduction.  

\subsection{Review of Ordinary BPS Invariants}\label{subsec:GVIntro}

Let $X$ be a Calabi-Yau threefold with curve class $\beta \in H_{2}(X, \Z)$.  The central quantity in the Maulik-Toda proposal defining the ordinary BPS states \cite{maulik_gopakumarvafa_2018}, is the Hilbert-Chow morphism\footnote{Here $M_{X}(0, 0, \beta, 1)$ is the moduli space of Simpson stable torsion sheaves on $X$ with generic polarization, and Chern character $(0,0,\beta,1) \in H^{2*}(X, \Z)$.  Moreover, $\Chow_{\beta}(X)$ is the Chow variety of effective curves in the class $\beta$.} 
\[\pi: M_{X}(0, 0, \beta, 1) \to \Chow_{\beta}(X)\]
along with the perverse sheaf of vanishing cycles $\phi$ on $M_{X}(0, 0, \beta, 1)$.  We then define the Maulik-Toda polynomial
\begin{equation}\label{eqn:DefnMTPoly}
\MT(\pi) = \sum_{i,j \in \Z} (-1)^{i}y^{j} \dim \HH^{i}\big(\Chow_{\beta}(X), ?{}^{\mathfrak{p}}?\mathcal{H}^{j}(R \pi_{*} \phi)\big)
\end{equation}
where $\HH$ denotes hypercohomology, and $?{}^{\mathfrak{p}}?\mathcal{H}(\cdot)$ are the cohomology sheaves with respect to the perverse t-structure.  For further details, see \cite{de_cataldo_decomposition_2009,maulik_gopakumarvafa_2018}.  By Verdier duality, $\MT(\pi)$ is a Laurent polynomial in $y$ invariant under $y \leftrightarrow y^{-1}$.  Therefore, we can write
\begin{equation}\label{eqn:DefnGVinvMT}
\MT(\pi) = \sum_{g \geq 0} n_{\beta}(g)\big(y^{\frac{1}{2}}+y^{-\frac{1}{2}}\big)^{2g}
\end{equation}
for uniquely determined integers $n_{\beta}(g)$, which we call the \emph{(ordinary) BPS invariants}.  

The example relevant to us is when the threefold $X$ is a local $K3$ or Abelian surface $S$ with effective class $\beta \in H_{2}(S, \Z)$.  In this case, one can define $X$ as the total space of a fibration $f:X \to (\Delta,0)$ by $K3$ or Abelian surfaces over a pointed disk $(\Delta,0)$ such that:
\begin{enumerate}
\item All fibers of $f$ are projective
\item $f^{-1}(0) = S$
\item The class $\beta \in H_{2}(S, \Z)$ does not deform algebraically off the central fiber to any order.  
\end{enumerate}  

Because curves in the class $\beta$ do not deform even scheme-theoretically off the central fiber, the theory localizes to studying sheaves on the surface $S$.  Let $M_{S}(0,\beta,1)$ be the moduli space of stable sheaves with generic polarization, and Chern character $(0, \beta,1) \in H^{2*}(S, \Z)$.  We know $M_{S}(0,\beta,1)$ is a smooth projective holomorphic symplectic variety, and we additionally assume that it is of $K3$-type (this is not true when $S$ is an Abelian surface, but by Corollary \ref{YoshCorr}, the fixed locus under the standard involution is a disjoint union of varieties of this type).    

Under the above assumptions, the Hilbert-Chow morphism
\[\pi : M_{S}(0,\beta,1) \to \Chow_{\beta}(S)\]
has $\Chow_{\beta}(S) \cong \P^{n}$ where $\dim(M_{S}(0,\beta,1))=2n$, by \cite{hwang_base_2007}.  Hence, $\pi$ is a Lagrangian fibration.  Since $M_{S}(0,\beta,1)$ is smooth, $\phi = \underline{\Q}[2n]$.  For all $0 \leq p,q \leq 2n$, the perverse Hodge numbers of $\pi$ are defined to be
\begin{equation}
?{}^{\mathfrak{p}}?h^{p,q}(\pi) := \dim \HH^{q-n}\big(\P^{n}, ?{}^{\mathfrak{p}}?\mathcal{H}^{p-n}(R \pi_{*} \underline{\Q}[2n])\big).
\end{equation}
\begin{thm}[{\cite[Thm.\@ 0.2]{shen_topology_2019}}]\label{thm:ShenYin}
Under the above hypotheses, the perverse Hodge numbers of $\pi$ equal the ordinary Hodge numbers of $M_{S}(0,\beta,1)$
\[?{}^{\mathfrak{p}}?h^{p,q}(\pi) = h^{p,q}(M_{S}(0,\beta,1)).\]
\end{thm}
Using (\ref{eqn:DefnMTPoly}) and (\ref{eqn:chiydefn}), the following is then a small computation.
\begin{cor}\label{cor:Shen-YinCOR}
The Maulik-Toda polynomial associated to $\pi$ can be expressed as the normalized $\chi_{y}$-genus of $M_{S}(0,\beta,1)$
\[\MT(\pi) = \overline{\chi}_{y}(M_{S}(0,\beta,1)).\]
\end{cor}

\subsection{\boldmath{$\tau$}-BPS Invariants of the Orbifold Kummer Surface}

There does not currently exist a definition of BPS invariants for orbifolds.  We give here a coarse\footnote{See Remark \ref{rmk:noTwSec}} definition in the case of the orbifold Kummer surface.  

Let $A$ be a polarized Abelian surface of type $(1,d)$ with $\beta_{d} \in H_{2}(A, \Z)$ the class of the polarization.  The orbifold Kummer surface is defined as the stack quotient $[A/\tau]$ where $\tau: A \to A$ is the standard involution.  

By the same argument as in Section \ref{subsec:GVIntro}, the Maulik-Toda proposal applied to a local orbifold Kummer surface will localize to a theory of sheaves on the surface itself.  A natural guess to define orbifold BPS states is to use the proper map
\begin{equation}\label{eqn:HilbChowinvOKS}
\pi_{d} : M_{A}(0,\beta_{d},1)^{\tau} \to \Chow_{\beta_{d}}(A)^{\tau}
\end{equation}
where the superscript $\tau$ denotes the fixed locus of the induced $\tau$ action.  

By Corollary \ref{YoshCorr}, the space $M_{A}(0,\beta_{d},1)^{\tau}$ is a disjoint union of smooth holomorphic symplectic varieties of $K3$-type.  We can decompose it, along with the invariant Chow variety, into connected components
\[M_{A}(0, \beta_{d}, 1)^{\tau} = \coprod_{k,l} M_{k,l} \,\,\,\,\,\,\,\,\,\,\,\,  \Chow_{\beta_{d}}(A)^{\tau} = \coprod_{l} B_{l}.\]
Restricting $\pi_{d}$ to $M_{k,l}$, we get surjective maps $\pi_{k,l} : M_{k,l} \to B_{l}$.  Since each $B_{l}$ is smooth, by \cite{hwang_base_2007} each $\pi_{k,l}$ is a connected Lagrangian fibration over a projective space $B_{l} = \P^{d_{l}}$ where $\dim(M_{k,l}) = 2d_{l}$.    

Associated to each $\pi_{k,l}$ with $\phi = \underline{\Q}_{M_{k,l}}[2d_{l}]$ we get the Maulik-Toda polynomial $\MT(\pi_{k,l})$ as in (\ref{eqn:DefnMTPoly}).  The Maulik-Toda polynomial associated to the map in (\ref{eqn:HilbChowinvOKS}) is then   
\[\MT(\pi_{d}) = \sum_{k,l} \MT(\pi_{k,l}).\]
\begin{defn}\label{defn:GVinvofOKS}
The $\tau$-BPS invariants of the orbifold Kummer surface $[A/\Z_{2}]$ are integers $\ndh$ defined for all $d \geq 0$ by 
\begin{equation}\label{eqn:MTdefn}
\MT(\pi_{d}) = \sum_{h=0}^{\infty} \ndh (y^{\frac{1}{2}}+ y^{-\frac{1}{2}})^{2h}.
\end{equation}
\end{defn}

\subsection{Proof of Proposition \ref{prop:KeyEnummPropp}}\label{proofofKeyEnummPropp}

For all $d \geq 0$, we denote by $M_{A}(1,0,-d)$ the moduli space of torsion-free sheaves on $A$ with Chern character $(1,0,-d) \in H^{2*}(A,\Z)$.  
\begin{lemma}\label{lemma:simplelemma}
There exists a canonical $\tau$-equivariant isomorphism 
\[
M_{A}(1,0,-d) \to \widehat{A} \times \Hilb^{d}(A).
\] 
\end{lemma}
\begin{proof}
A sheaf $\mathcal{F} \in M_{A}(1,0,-d)$ can be uniquely expressed as $\mathcal{F} = L_{0} \otimes \mathcal{I}_{d}$ where $L_{0} \in \widehat{A}$ is a degree-zero line bundle and $\mathcal{I}_{d}$ is the ideal sheaf of a zero-dimensional subscheme of length $d$.  The $\tau$-equivariance is immediate.  
\end{proof}
Let $A_{0} = E \times F$ be a product Abelian surface with $E$ and $F$ two elliptic curves.  Denote by $\sigma$ the class of $E \times \{\text{pt}\}$, and by $f$ the class of $\{\text{pt}\} \times F$.  Lemma \ref{lemma:simplelemma} reduces the proof of the proposition to constructing a $\tau$-equivariant deformation equivalence 
\begin{equation}\label{eqn:defeqqq}
M_{A}(1,0,-d) \to M_{A}(0, \beta_{d}, 1).
\end{equation}
We do so by deforming to $A_{0}$, as summarized in the diagram  
\begin{equation}\label{eqn:defeqdiag}
\begin{tikzcd}
M_{A}(1,0,-d)  \arrow[swap]{d}{} & M_{A}(0, \beta_{d}, 1) \\
M_{A_{0}}(1,0,-d) \arrow{r}{\sim}& M_{A_{0}}(0, \sigma +df, 1) \arrow{u}{}
\end{tikzcd}
\end{equation}
Here, as we will explain, the vertical arrows are $\tau$-equivariant deformation equivalences and the horizontal arrow is a $\tau$-equivariant isomorphism induced from the relative Fourier-Mukai functor applied to the natural projection $A_{0} \to E$.  

\begin{lemma}
The vertical arrows in diagram (\ref{eqn:defeqdiag}) are $\tau$-equivariant deformation equivalences.  
\end{lemma}
\begin{proof}
By Proposition 4.12 of \cite{yoshioka_moduli_2001}, we have a deformation equivalence 
\[
M_{A}(1,0,-d) \to M_{A_{0}}(1,0,-d).
\]
The proof involves a family of polarized Abelian surfaces.  But clearly we have a fiberwise action on the family using the standard involution on each Abelian surface.  It is in this sense that the deformation equivalence is $\tau$-equivariant.  Precisely the same argument holds for $M_{A_{0}}(0, \sigma +df, 1) \to M_{A}(0, \beta_{d}, 1)$.  
\end{proof}

Consider the following diagram
\begin{equation}\label{dfsds}
\begin{tikzcd}
& A_{0} \times_{E} A_{0} \arrow{rd}{q} \arrow{ld}[swap]{p} & \\
A_{0} & & A_{0}
\end{tikzcd}
\end{equation}
where $p,q$ are the two canonical projections.  Let $P$ denote the universal Poincar\'{e} line bundle on $A_{0} \times_{E} A_{0}$, and define the relative Fourier-Mukai functor
\[
\Phi_{P}: D^{b}(A_{0}) \to D^{b}(A_{0})
\]
by $\mathcal{E} \mapsto Rq_{*}\big(P \otimes p^{*}\mathcal{E}\big)$.  For more details on relative Fourier-Mukai in more generality, see Section 3.2 of \cite{yoshioka_moduli_2001} or Chapter 6 of \cite{bartocci_fourier-mukai_2009}.  

\begin{lemma}
The relative Fourier-Mukai functor $\Phi_{P}$ is $\tau$-equivariant.  More speecifically, for all $\mathcal{E} \in D^{b}(A_{0})$
\[
\tau^{*} \big( \Phi_{P}(\mathcal{E}) \big) \cong \Phi_{P} \big( \tau^{*}(\mathcal{E}) \big).
\]
\end{lemma}
\begin{proof}
The key is that the Poincar\'{e} bundle is $\tau_{\Delta}$-invariant; that is, there exists an isomorphism $P \cong \tau_{\Delta}^{*} P$ where $\tau_{\Delta}$ is the induced diagonal action on $A_{0} \times_{E} A_{0}$.  This is because (see Definition 6.14 of \cite{bartocci_fourier-mukai_2009})
\[
P = \mathcal{O}_{A_{0} \times_{E} A_{0}}\big(D + D' -\Delta\big)
\]
where $D = A_{0} \times \{0\}$, $D'=\{0\} \times A_{0}$ and $\Delta$ is the diagonal divisor in $A_{0} \times_{E} A_{0}$.  All three of these divisors are $\tau_{\Delta}$-invariant.  The remainder of the argument follows from straightforward functorial properties of $\tau$, $p$, and $q$.

\end{proof}

\begin{cor}
The relative Fourier-Mukai functor $\Phi_{P}$ induces a $\tau$-equivariant isomorphism
\[
M_{A_{0}}(1,0,-d) \to M_{A_{0}}(0,\sigma +df,1).
\]
\end{cor}
\begin{proof}
The isomorphism is a particular example of Theorem 3.15 in \cite{yoshioka_moduli_2001}, noting that because our Chern character is primitive, there are no strictly semistable sheaves.  The $\tau$-equivariance follows from the previous lemma.  
\end{proof}

This completes the proof of Proposition \ref{prop:KeyEnummPropp}.

\subsection{Proof of Theorem \ref{prop:KKVSYProp}}\label{ShenYinproof}
By Corollary \ref{cor:Shen-YinCOR} we have 
\begin{equation}\label{eqn:MTchiy}
\MT(\pi_{d}) = \sum_{k,l} \overline{\chi}_{y}(M_{k,l}) = \overline{\chi}_{y}\big(M_{A}(0,\beta_{d},1)^{\tau}\big)
\end{equation}
where the latter equality holds because the ordinary Hodge numbers are additive under disjoint unions.  By Corollary \ref{YoshCorr}, we have
\[\overline{\chi}_{y}\big(M_{A}(0, \beta_{d}, 1)^{\tau}\big) = 16 \, \overline{\chi}_{y}\big(\Hilb^{d}(A)^{\tau}\big)\]
since the $\chi_{y}$-genus is a deformation invariant.  Therefore
\[
\tfrac{1}{16} \MT(\pi_{d}) = \overline{\chi}_{y}\big(\Hilb^{d}(A)^{\tau}\big).
\]
The result then follows using Proposition \ref{prop:chiyrefinement} along with the expression for $Z_{A, \tau}^{-1}$ from Table \ref{table:TAB1}.

\subsection{Proof of Theorem \ref{thm:MainEmThm}}\label{ratlcurveproof}
In this section we will prove Theorem \ref{thm:MainEmThm} on the Yau-Zaslow formula for the orbifold Kummer surface.  We first need to establish a few lemmas.
\begin{lemma}\label{lemma:BeauLemmAbSurf}
Let $C$ be an integral curve with an involution $\tau:C \to C$.  If the quotient $C/\tau$ is not rational, then $e(\CompJac(C)^{\tau})=0$. 
\end{lemma}
\begin{proof}
Let $\eta : \widetilde{C} \to C$ and $\delta : \widetilde{C/\tau} \to C/\tau$ be the corresponding normalization maps, and $\alpha: C \to C/\tau$ the quotient.  We get a commuting diagram
\[ \begin{tikzcd}
\widetilde{C} \arrow{r}{\eta} \arrow[swap]{d}{\widetilde{\alpha}} & C \arrow{d}{\alpha} \\
\widetilde{C/\tau} \arrow{r}{\delta}& C/\tau
\end{tikzcd}
\]
where the map $\widetilde{\alpha}$ is defined by the universal property of normalizations: since $\alpha$ is surjective and $C/\tau$ is integral, $\alpha$ factors uniquely via a map $C \to \widetilde{C/\tau}$.  We get $\widetilde{\alpha}$ by composing with $\eta$.  By pullback we get a short exact sequence of Abelian groups
\[0 \to \Jac(\widetilde{C/\tau}) \to \Jac(\widetilde{C}) \to \text{Prym}(\widetilde{\alpha}) \to 0\]
where $\text{Prym}(\widetilde{\alpha})$ is the Prym variety associated to $\widetilde{\alpha}$.  Since $\widetilde{\alpha}$ is $\tau$-equivariant with the trivial action on $\widetilde{C/\tau}$, by restricting to the fixed locus we get an inclusion
\begin{equation}\label{eqn:AbVarSubgrp}
\Jac(\widetilde{C/\tau}) \subset \Jac(\widetilde{C})^{\tau}.
\end{equation}
Pulling back via $\eta$, we get a short exact sequence
\[0 \to G \to \Jac(C) \to \Jac(\widetilde{C}) \to 0\]
where $G$ is a product of additive and multiplicative groups.  By \cite[Prop.2.2]{beauville_counting_1999_TWO}, this sequence splits canonically so that $\Jac(\widetilde{C}) \subset \Jac(C)$ corresponds to the subgroup of line bundles on $C$ pushed forward from line bundles on $\widetilde{C}$.  Because $\eta$ is $\tau$-equivariant, this descends to an inclusion 
\[\Jac(\widetilde{C})^{\tau} \subset \Jac(C)^{\tau}.\]
By \cite[Lem.2.1]{beauville_counting_1999_TWO} we thereby get a free action of $\Jac(\widetilde{C})^{\tau}$ on the invariant compactified Jacobian $\CompJac(C)^{\tau}$.  But if $C/\tau$ is not rational, we know by (\ref{eqn:AbVarSubgrp}) that a positive-dimensional Abelian variety therefore acts freely on $\CompJac(C)^{\tau}$.  If a finite group of order $n$ acts freely on $\CompJac(C)^{\tau}$, then $n$ must divide $e(\CompJac(C)^{\tau})$.  Because an abelian variety contains cyclic subgroups of all orders, $e(\CompJac(C)^{\tau})=0$.  
\end{proof}

\begin{lemma}\label{lemma:EulCharLemma}
If $f: X \to Y$ is a surjective morphism of projective algebraic varieties and all fibers have vanishing Euler characteristic, then $e(X)=0$.
\end{lemma}
\begin{proof}
Consider first the case that $f$ is a topological locally-trivial fibration with fiber $F$.  A well-known result is that $e(X) = e(F) \cdot e(Y)$ from which the lemma follows.  In the more general case, take a stratification of $Y$ such that $f$ is a topological locally-trivial fibration over each strata.  The excision property of the Euler characteristic then reduces the problem to the previous case.
\end{proof}

\begin{proof}[Proof of Theorem \ref{thm:MainEmThm}]
The genus $h=0$ specialization of the Maulik-Toda polynomial is $y=-1$, as we see from (\ref{eqn:MTdefn}).  By the proof of Theorem \ref{prop:KKVSYProp},
\begin{equation}
\begin{split}
\ndo & = \mathsf{MT}_{-1}(\pi_{d}) \\
& = e\big( M_{A}(0,\beta_{d},1)^{\tau}\big) = 16 \, e\big( \Hilb^{d}(A)^{\tau} \big).
\end{split}
\end{equation}
This proves the first assertion, noting the expression for $Z_{A, \tau}$ (see Table \ref{table:TAB1}).   

For the final claim we assume $A$ is a $(1,d)$-polarized Abelian surface with Picard rank one, and $\beta_{d}$ the unique primitive generator.  In this case, 
\[M_{A}(0,\beta_{d},1) \cong \RelJac\]
where $\RelJac \to \Chow_{\beta_{d}}(A)$ is the relative compactified Jacobian of degree $d+1$, which is the arithmetic genus of the class $\beta_{d}$.  The fiber over a curve $C \in \Chow_{\beta_{d}}(A)$ is $\overline{\Jac}^{\,d+1}(C)$ parameterizing rank $1$ torsion-free sheaves of degree $d+1$ on $C$.  We define
\[
\Pi \subset \Chow_{\beta_{d}}(A)^{\tau}
\]
to be the set of $\tau$-invariant curves in the class $\beta_{d}$ with rational quotient.  The set $\Pi$ is finite because if it were not, the singular $K3$ surface $A/\tau$ would contain a positive-dimensional family of rational curves, which cannot occur.  Let $Y$ be the open subvariety of the invariant Chow variety parameterizing $\tau$-invariant curves with non-rational quotient.
\[Y = \Chow_{\beta_{h}}(A)^{\tau} - \Pi.\]
Given the map $\pi_{d}: \RelJac^{\tau} \to \Chow_{\beta_{d}}(A)^{\tau}$, we get the following decomposition of the total space
\[\RelJac^{\tau} = \pi_{d}^{-1}(Y) \cup \pi_{d}^{-1}(\Pi).\]
Since $\Pi$ is finite, $\pi_{d}^{-1}(\Pi)$ is a closed subvariety of $\RelJac^{\tau}$ whose compliment is $\pi_{d}^{-1}(Y)$.  Thus, by the excision property of the Euler characteristic
\[e(\RelJac^{\tau}) = e(\pi_{d}^{-1}(Y)) + e(\pi_{d}^{-1}(\Pi)).\]

The fiber of $\pi_{d}$ over an invariant curve $C$ is $\CompJac(C)^{\tau} \cong \CompJac^{\, d+1}(C)^{\tau}$, where the isomorphism is twisting by a fixed $\tau$-invariant line bundle of degree $d+1$.  Therefore, all fibers of the restricted family $\pi_{d}^{-1}(Y) \to Y$ have vanishing Euler characteristic by Lemma \ref{lemma:BeauLemmAbSurf}, and $e(\pi_{d}^{-1}(Y))=0$ by Lemma \ref{lemma:EulCharLemma}.  Finally, we have 
\[\pi_{d}^{-1}(\Pi) = \coprod_{C \in \Pi} \CompJac(C)^{\tau},\]
from which it now follows that 
\[e(\RelJac^{\tau}) = \sum_{C \in \Pi} e(\CompJac(C)^{\tau}).\]
Because $\ndo = e(\RelJac^{\tau})$, this completes the proof.  
\end{proof}

\subsection{Proof of Conjecture \ref{conj:CompJacConj} for Smooth Curves}\label{subsection:pf of conj for sm curves}

Let $C$ be a smooth curve of genus $g$ with an involution $\tau: C \to C$.  Let $\alpha : C \to C/\tau$ be the quotient map, and let $h$ be the genus of the smooth curve $C/\tau$.  Just as in the proof of Lemma \ref{lemma:BeauLemmAbSurf}, we get a short exact sequence
\[
0 \to \Jac(C/\tau) \to \Jac(C) \to \text{Prym}(\alpha) \to 0
\]
where $\text{Prym}(\alpha)$ is an Abelian variety of dimension $g-h$ called the Prym variety.  There is a canonical lift of the $\tau$ action to $\Jac(C)$ by pullback which acts trivially on $\Jac(C/\tau)$.  Moreover, the induced action on $\text{Prym}(\alpha)$ is by the standard invoultion $a \mapsto -a$, which has $2^{2g-2h}$ isolated fixed points.  Therefore
\[
\Jac(C)^{\tau} = \coprod_{i=1}^{2^{2g-2h}} \Jac(C/\tau)
\]
If $C/\tau$ is a rational curve, then $h=0$ and $\Jac(C/\tau)$ is a point, which proves the conjecture in this case.

\appendix
\section{Modular Forms and Eta Products}\label{modular}

This appendix will be devoted to giving a brief overview of the modular objects relevant to our results.  An excellent reference is Chapters 1 and 2 of \cite{kohler_eta_2011}.  We are interested in modular forms of integral or half-integral weight with multiplier system for the congruence subgroup 
\[\Gamma_{0}(N) := \bigg\{
\begin{pmatrix}
a & b  \\
c & d 
\end{pmatrix}
\in \SL_{2}(\Z) \, \bigg| \, c \equiv 0 \Mod{N} \bigg\} \subset \SL_{2}(\Z)\]
for an integer $N \geq 1$.  A multiplier system on $\Gamma_{0}(N)$ is a function $v : \Gamma_{0}(N) \to \C^{*}$ satisfying some consistency conditions.  We will not need the details, so we refer the reader to \cite[Sec.\@ 2.6]{iwaniec_topics_1997}.

Let $\Hf = \{ z \in \C \, | \, \Im(z) > 0 \}$ be the upper-half plane in $\C$.  
\begin{defn}
A holomorphic function $f: \Hf \to \C$ is called a \emph{modular form} of weight $k \in \R$ and multiplier system $v$ on $\Gamma_{0}(N)$ if $f$ is holomorphic at all cusps $\Q \cup \{\infty\}$, and if for all $L \in \Gamma_{0}(N)$, $f$ transforms as
\begin{equation}\label{modfrmtrans}
f(L \tau) = f\bigg(\frac{a \tau + b}{c \tau + d}\bigg) = v(L) (c \tau +d)^{k} f(\tau), \,\,\,\,\,\,\,
L=
\begin{pmatrix}
a & b  \\
c & d 
\end{pmatrix}.
\end{equation}
We call $f$ a \emph{cusp form} if additionally, $f$ vanishes at all cusps.    
\end{defn}

We employ the change of variables $q = \exp(2 \pi i \tau)$ and with it, the abuse of notation writing $f(\tau)$ and $f(q)$ interchangeably.  The fundamental building block of a large class of modular forms is the Dedekind eta function (or just eta function)
\[\eta(q) = q^{\frac{1}{24}} \prod_{n=1}^{\infty}(1-q^{n}),\]
which is a modular form of weight $\frac{1}{2}$ and multiplier system $v_{\eta}$ on $\SL_{2}(\Z)$; see \cite[Sec.\@ 2.8]{iwaniec_topics_1997} where $v_{\eta}$ is given explicitly.\footnote{Equivalently, $\eta(q)$ is a modular form of weight $\frac{1}{2}$ on the metaplectic double cover $\text{Mp}_{2}(\Z)$ of $\SL_{2}(\Z)$.}  
\begin{defn}
An \emph{eta product} of level $N \geq 1$ is a function $f : \Hf \to \C$ of the form
\[f(q) = \prod_{m|N} \eta(q^{m})^{a_{m}}\]
such that $a_{m} \in \Z$ (possibly negative, or zero) for all $m | N$, and where the product is over positive divisors of $N$.  
\end{defn}
From the modular properties of the Dedekind eta function, one can show that an eta product $f$ of level $N$ transforms as a modular form on $\Gamma_{0}(N)$ of weight
\[k = \frac{1}{2} \sum_{m|N} a_{m} \in \tfrac{1}{2}\Z,\]
and with multiplier system
\[v_{f}(L) = \prod_{m|N} \bigg(v_{\eta}
\begin{pmatrix}
a & mb  \\
c/m & d 
\end{pmatrix}\bigg)^{a_{m}}, \,\,\,\,\,\,\,\,\,\,\,\,\,\,\,\,\,\, 
L=
\begin{pmatrix}
a & b  \\
c & d 
\end{pmatrix}\]
When we say ``transforms as" we mean that $f$ satisfies (\ref{modfrmtrans}) for all $L \in \Gamma_{0}(N)$.  An eta product $f$ is automatically holomorphic on $\Hf$.  This is because the form of $\eta(q)$ indicates that any poles of $f$ must occur at $q=0$ or $|q|=1$.  All that is left to consider is when an eta product is holomorphic at the cusps.  The following proposition gives necessary and sufficient conditions.  
\begin{prop}[{\cite[Cor.\@ 2.3]{kohler_eta_2011}}]\label{prop:KohlerProp}
An eta product $f$ of level $N$ is holomorphic at the cusps if and only if the following holds for all positive divisors $c$ of $N$
\[\sum_{m|N} \frac{(\gcd(c,m))^{2})}{m} a_{m} \geq 0.\]
Moreover, $f$ vanishes at all cusps if and only if each inequality is strict.  An eta product is therefore a modular form of weight $k$ for $\Gamma_{0}(N)$ if and only if each inequality is satisfied, and it is a cusp form if and only if each is strictly satisfied.  
\end{prop}


\bibliographystyle{alpha}
\bibliography{SymplecticActionsAbelianSurface,SecondBib}

\begin{thebibliography}{BOPY18}

\bibitem[BBR09]{bartocci_fourier-mukai_2009}
Claudio Bartocci, Ugo Bruzzo, and Daniel~Hernández Ruipérez.
\newblock {\em Fourier-{Mukai} and {Nahm} {Transforms} in {Geometry} and
  {Mathematical} {Physics}}.
\newblock Progress in {Mathematics}. Birkhäuser Basel, 2009.

\bibitem[Bea99]{beauville_counting_1999_TWO}
Arnaud Beauville.
\newblock Counting rational curves on {$K3$} surfaces.
\newblock {\em Duke Mathematical Journal}, 97(1):99--108, 1999.

\bibitem[BG20]{bryan_g-fixed_2020_TWO}
Jim Bryan and Ádám Gyenge.
\newblock {$G$}-fixed {Hilbert} schemes on {$K3$} surfaces, modular forms, and
  eta products.
\newblock {\em arXiv:1907.01535 [math]}, 2020.

\bibitem[BL99]{bryan_enumerative_1999}
Jim Bryan and Naichung Leung.
\newblock The {Enumerative} {Geometry} of {K3} {Surfaces} {And} {Modular}
  {Forms}.
\newblock {\em Journal of the American Mathematical Society}, 13, December
  1999.

\bibitem[BOPY18]{bryan_curve_2018}
Jim Bryan, Georg Oberdieck, Rahul Pandharipande, and Qizheng Yin.
\newblock Curve counting on abelian surfaces and threefolds.
\newblock {\em Algebraic Geometry}, 5(4):398--463, 2018.

\bibitem[BP]{bryan_counting_nodate}
Jim Bryan and Stephen Pietromonaco.
\newblock Counting {Invariant} {Curves}: {A} {Theory} of {Gopakumar}-{Vafa}
  {Invariants} for {Calabi}-{Yau} {Threefolds} with an {Involution}.
\newblock (In preparation).

\bibitem[dCM09]{de_cataldo_decomposition_2009}
Mark de~Cataldo and Luca Migliorini.
\newblock The decomposition theorem, perverse sheaves and the topology of
  algebraic maps.
\newblock {\em Bulletin of the American Mathematical Society}, 46(4):535--633,
  2009.

\bibitem[EZ85]{eichler_theory_1985}
Martin Eichler and Don Zagier.
\newblock {\em The {Theory} of {Jacobi} {Forms}}.
\newblock Progress in {Mathematics}. Birkhäuser Basel, 1985.

\bibitem[FGvS99]{fantechi_euler_1999}
Barbara Fantechi, Lothar Göttsche, and Duco van Straten.
\newblock Euler number of the compactified {Jacobian} and multiplicity of
  rational curves.
\newblock {\em Journal of Algebraic Geometry}, 8:115--133, 1999.

\bibitem[Fuj88]{fujiki_finite_1988}
Akira Fujiki.
\newblock Finite {Automorphism} {Groups} of {Complex} {Tori} of {Dimension}
  {Two}.
\newblock {\em Publ. Res. Inst. Math. Sci.}, 24(1):1--97, 1988.

\bibitem[Gö90]{gottsche_betti_1990}
Lothar Göttsche.
\newblock The {Betti} numbers of the {Hilbert} scheme of ponts on a smooth
  projective surface.
\newblock {\em Mathematische Annalen}, 286(1-3):193--208, 1990.

\bibitem[Hwa07]{hwang_base_2007}
Jun-Muk Hwang.
\newblock Base manifolds for fibrations of projective irreducible symplectic
  manifolds.
\newblock {\em Inventiones mathematicae}, 174(3), November 2007.

\bibitem[Iwa97]{iwaniec_topics_1997}
Henryk Iwaniec.
\newblock {\em Topics in {Classical} {Automorphic} {Forms}}, volume~17 of {\em
  Graduate {Studies} in {Mathematics}}.
\newblock American Mathematical Society, 1997.

\bibitem[Kat08]{katz_genus_2008}
S.~Katz.
\newblock Genus zero {Gopakumar}-{Vafa} invariants of contractible curves.
\newblock {\em Journal of Differential Geometry}, 79(2):185--195, June 2008.

\bibitem[KKV99]{katz_m_1999}
Sheldon~H. Katz, Albrecht Klemm, and Cumrun Vafa.
\newblock M theory, topological strings and spinning black holes.
\newblock {\em Adv.Theor.Math.Phys.}, 3:1445--1537, October 1999.
\newblock arXiv:hep-th/9910181.

\bibitem[Kö11]{kohler_eta_2011}
Günter Köhler.
\newblock {\em Eta {Products} and {Theta} {Series} {Identities}}.
\newblock Springer {Monographs} in {Mathematics}. Springer-Verlag, Berlin
  Heidelberg, 2011.

\bibitem[MT18]{maulik_gopakumarvafa_2018}
Davesh Maulik and Yukinobu Toda.
\newblock Gopakumar–{Vafa} invariants via vanishing cycles.
\newblock {\em Inventiones mathematicae}, 213(3):1017--1097, September 2018.

\bibitem[PT16]{Pandharipande:2014qoa}
R.~Pandharipande and R.P. Thomas.
\newblock {The Katz-Klemm-Vafa conjecture for {$K3$} surfaces}.
\newblock {\em Forum Math. Pi}, 4:e4, 2016.

\bibitem[Ros12]{rose_counting_2012}
Simon Rose.
\newblock Counting {Hyperelliptic} curves on {Abelian} surfaces with
  {Quasi}-modular forms.
\newblock {\em arXiv:1202.2094 [math]}, April 2012.
\newblock PhD thesis, University of British Columbia 2012.

\bibitem[SY19]{shen_topology_2019}
Junliang Shen and Qizheng Yin.
\newblock Topology of {Lagrangian} fibrations and {Hodge} theory of
  hyper-{Kahler} manifolds.
\newblock {\em arXiv:1812.10673 [math]}, May 2019.

\bibitem[Vol14]{volpato_symmetries_2014_TWO}
Roberto Volpato.
\newblock On symmetries of {$\mathcal{N}=(4,4)$} sigma models on {$T^{4}$}.
\newblock {\em Journal of High Energy Physics}, 2014(8), 2014.

\bibitem[Yos01]{yoshioka_moduli_2001}
Kota Yoshioka.
\newblock Moduli spaces of stable sheaves on abelian surfaces.
\newblock {\em Mathematische Annalen}, 321(4):817--884, 2001.

\bibitem[YZ96]{yau_bps_1996}
Shing-Tung Yau and Eric Zaslow.
\newblock {BPS} states, string duality, and nodal curves on {K3}.
\newblock {\em Nuclear Physics B}, 471(3):503--512, 1996.

\bibitem[Zha21]{zhan_counting_2021}
Sailun Zhan.
\newblock Counting {Rational} {Curves} on {K3} {Surfaces} {With} {Finite}
  {Group} {Actions}.
\newblock {\em International Mathematics Research Notices}, (rnaa320), January
  2021.
\newblock arXiv:1907.03330.

\end{thebibliography}

\smallskip
\smallskip

\noindent Department of Mathematics, University of British Columbia, 1984 Mathematics Road, Vancouver, B.C., Canada V6T1Z2.
\smallskip

\noindent email: {\tt spietro@math.ubc.ca}

\end{document}